\documentclass[article, letterpaper, oneside]{aram}

\usepackage[color]{aram}

\usepackage[]{aramtikz}

\graphicspath{ {./figures/} }

\setupheaders

\partialordergenerator{f}{F} 
\partialordergenerator{fp}{F'} 
\partialordergenerator{R}{\msR} 
\partialordergenerator{poc}{\mcC} 
\partialordergenerator{p}{P} 
\partialordergenerator{B}{B} 
\partialordergenerator{Div}{|} 

\DeclareMathOperator{\support}{supp}
\renewcommand\supp{\support}

\newcommand\xequal[1]{\stackrel{\mathclap{\normalfont\mbox{#1}}}{=}}

\title{Thin MC left regular bands}
\author{Aram Dermenjian}
\email{aram.dermenjian.math@gmail.com}
\thanks{The author was funded by NSERC and Heilbronn Institute for Mathematical Research.}

\begin{document}
%
%

\begin{abstract}
    We define MC left regular bands and study their adjacency graphs.
    We prove that for thin MC left regular bands, the adjacency graph is particularly nice and is represented by edge labeled graphs where every simple cycle has an even number of edges.
    Conversely, we define a set of graphs which we call thin LRB graphs which encode rank two thin MC left regular bands.
    Along the way, we provide a criterion for showing when the face poset of a left regular band is a meet-semilattice.
\end{abstract}

\mytitle

\section{Introduction}
Bands are idempotent semigroups and have been studied since at least 1940 when Fritz Klein-Barmen started studying what happens to semilattices if the property of commutativity was taken away in \cite{KleinBarmen1940}.
The first use of the term ``band'' appeared in \cite{Clifford1954} when Clifford started studying bands in more depth.
Left regular bands are bands with the additional identity $xyx = xy$.
The study of left regular bands came shortly after Klein-Barmen when in 1947 Sch\"utzeberger started studying non-commutative lattices himself in \cite{Schuetzenberger1947}.

In recent years, there has been a push to understand the structure of left regular bands.
One of the reasons for this is due to the number of various combinatorial objects which contain the structure of a left regular bands.
Some of these objects include
Tsetlin's library \cite{BidigareHanlonRockmore1999, Brown2000},
(real central) hyperplane arrangements \cite{BidigareHanlonRockmore1999, Brown2000, BrownDiaconis1998},
complex hyperplane arrangements \cite{Bjoerner2008},
matroids \cite{Brown2000},
oriented matroids \cite{BrownDiaconis1998},
interval greedoids \cite{Bjoerner2008},
oriented interval greedoids \cite{SaliolaThomas2011}, and
certain bicolored graphs \cite{ChungGraham2012}.
Another reason for the push to understand left regular bands is due to their rich representation theory
and applications to probability theory through random walks and Markov chains(see \cite{AguiarMahajan2010, AyyerSchillingSteinbergThiery2015, Brown2000, Brown2004, MargolisSaliolaSteinberg2015a, Saliola2006, Saliola2009, Steinberg2016}).
Moreover, from an order theoretic perspective, Green's $\msR$-preorder is an order precisely when a band is a left regular band giving us a rich combinatorial way to study these semigroups.

In this article, we introduce an adjacency graph on an arbitrary left regular band by letting the vertices be the chambers and having an edge between two chambers if they both cover a common facet in the face poset (the dual of Green's $\msR$-order).
We study a particular subset of left regular bands which we call \emph{thin MC left regular bands}.
Thin MC left regular bands are generalizations of real hyperplane arrangements, oriented matroids and certain stratifications of complex hyperplane arrangements.
A thin MC left regular band is a left regular band such that every facet is covered by precisely two chambers (thin), whose face poset is a meet-semilattice (M) and whose adjacency graph is connected (C).
We endow the adjacency graph with edge labels whose labeling is given by the support lattice.
This labeling enables us to understand the structure of the adjacency graph much more clearly and allows us to give a nice description of the minimal cycles in the adjacency graph given in the following theorem which we prove in \autoref{sec:mc-left-regular-bands}.
\begin{reptheorem}{cor:to-graph}
    Let $S$ be a finite thin MC left regular band.
    Then the adjacency graph $\Gamma_\mcC$ is an edge-labeled graph where every simple cycle has an even number of edges labeled by elements with the same support and each label appears exactly once.
\end{reptheorem}

Due to the structure of the adjacency graph that appears for thin MC left regular bands we define a graph called a thin left regular band graph to encapsulate this structure.
Thin left regular band graphs are edge labeled graphs in which every label appears exactly once and every minimal cycle has an even number of edges.
We then construct a (rank $2$ thin MC) left regular band from the thin left regular band graph by explicitly describing the multiplication between elements.
This allows us to have a bijection between thin left regular band graphs and rank $2$ thin MC left regular bands.
We end this paper by giving future directions of this project and some open problems which we hope to solve in upcoming work.

\section{Left regular bands}\label{sec:left-regular-bands}
We start by giving basic definitions.
The interested reader is referred to \cite{CliffordPreston2014} for more details.
A \defn{semigroup $(S, \cdot)$} is a set $S$ with a binary operation $\cdot: S \times S \to S$ that satisfies the associativity property.
If the binary operation is evident, by abuse of notation we let $S$ denote the semigroup.
Additionally, when using $\cdot$ as the binary operation, for $a, b \in S$ we let $ab$ be an equivalent way of writing $a \cdot b$.

This article studies a particular type of semigroup called a left regular band.
A \defn{band $S$} is a semigroup in which every element is idempotent, \ie $x^2 = x$ for all $x \in S$.
A \defn{left regular band $S$} is a band such that $xyx = xy$ for all $x, y \in S$.
Frequently we will have an identity element in our semigroup in which case our semigroup will be referred to as a \defn{monoid}.
If our left regular band has an identity we will call it a \defn{left regular band monoid}.

\subsection{Example - Semilattices}
Our first main example will be meet-semilattices.
For further details on lattice theory and order theory in general the interested reader is referred to \cite{LatticeTheory}.

Recall that a \defn{poset $(P, \pleq)$} is a set $P$ with a partial order $\pleq$ which is reflexive, anti-symmetric and transitive.
For any two elements $x, y \in P$ let $x \pmeet y$ denote the \defn{meet} (the greatest lower bound) of $x$ and $y$ if it exists and let $x \pjoin y$ denote the \defn{join} (the least upper bound) of $x$ and $y$ if it exists.
If the meet (join) exists for every pair of elements $x, y \in P$ then we say the poset is a \defn{meet-semilattice} (\defn{join-semilattice}).
A poset is called a \defn{lattice} if it is both a meet-semilattice and a join-semilattice.
Both meet-semilattices and join-semilattices are examples of left regular bands.
In fact, since $x \pmeet y = y \pmeet x$, they are commutative left regular bands.
It is well known that a left regular band is commutative if it is isomorphic to a meet-semilattice.

As semilattices are used heavily in this paper, we define some further definitions pertaining to semilattices.
An \defn{interval $[x,y]_P$} in a poset $(P, \pleq)$ is the set of elements weakly between $x$ and $y$, \ie $[x,y]_P = \set{z \in P}{x \pleq z \pleq p}$.
We say that an element $x$ is \defn{covered by} $y$ (or that $y$ \defn{covers} $x$), denoted $x \pcover y$, if $x \pleq y$ and $[x,y]_P = \left\{ x,y \right\}$.
An \defn{atom} in a (finite) meet-semilattice is an element which covers the unique minimal element.
Similarly, a \defn{coatom} in a (finite) join-semilattice is an element which is covered by the unique maximal element.
Finally, the \defn{Hasse diagram} of a poset is the graph with vertex set $P$ and edge set $(x, y)$ whenever $x \pcover y$.
By convention, we draw $x$ below $y$ in the Hasse diagram in order to visualize the order.

Two examples of posets can be viewed at the end of this section in \autoref{ex:hyp-arr}.
The graphs given are the Hasse diagrams of two different posets which we define later.
On the left is the face poset which is a meet-semilattice whose atoms are given by the vertices $F_i$.
The interval $[0, R_3]$ in this poset is the set of elements $\left\{ 0, F_2, F_3, R_3 \right\}$.
On the right is the support lattice which is both a meet-semilattice and a join-semilattice.

\subsection{Green's \mt{\msR}{R}-order}
Recall that a preorder is a relation which is reflexive and transitive, \ie a partial order without anti-symmetry.
Given a band $S$, \defn{Green's $\msR$-preorder} on $S$ is the preorder where for all $x, y \in S$
\[
    x \Rleq y \text{ if and only if } x = yx.
\]
This preorder is transitive for all semigroups and reflexive for bands.
It turns out that this preorder is a partial order precisely when our band is a left regular band.
As this paper exclusively studies left regular bands, we will call this partial order \defn{Green's $\msR$-order}.

\subsection{Face poset}
Due to the connections with geometric structures, our focus will be on a slightly different order, the dual of Green's $\msR$-order for left regular bands.
Let $S$ be a band with elements $x$ and $y$ and let $\fleq$ be the partial order
\[
    y \fleq x \text{ if and only if } x = yx \text{ if and only if } y \Rgeq x.
\]
We let $(S, \fleq)$ denote this poset and call it the \defn{face poset}.

The face poset has a unique minimal element whenever the left regular band has an identity, but it rarely has a unique maximal element.
The maximal elements are known as \defn{chambers} and the elements covered by the chambers are known as \defn{facets}.
Let $\mcC_S$ denote the set of all chambers in $(S, \fleq)$.
If there is no ambiguity, we use $\mcC$ instead of $\mcC_S$.
The \defn{rank} of a left regular band is the length of the shortest path from a minimal element to a chamber.

It turns out that there is a second equivalent way to define a chamber:
\begin{lemma}[\cite{Brown2000}]
    \label{lem:equiv-chamber}
    Let $S$ be a left regular band.
    The following are equivalent:
    \begin{itemize}
        \item $C$ is a chamber, \ie is a maximal element in $(S, \fleq)$, and
        \item for all $x \in S$, $Cx = C$.
    \end{itemize}
\end{lemma}

We give some properties for the face poset in the following lemma which will be useful later on.
We provide a proof for each property even though these properties are already well-known.
\begin{lemma}
    \label{lem:fleq}
    For all $x, y, z \in S$ where $S$ a left regular band.
    Then
    \begin{enumerate}
        \item $xy \fleq y$ implies $x \fleq y$,
        \item $x \fleq xy$, and
        \item $x \fleq z$ and $y \fleq z$ implies $xy \fleq z$.
    \end{enumerate}
\end{lemma}
\begin{proof}
    Suppose $x$, $y$, and $z$ are elements in a left regular band $S$.
    \begin{enumerate}
        \item $xy \fleq y$ implies $y = (xy)y = x(yy) = xy$ implying $x \fleq y$.
        \item $x(xy) = (xx)y = xy$ implying $x \fleq xy$.
        \item As $x \fleq z$ and $y \fleq z$ we know $xz = z$ and $yz = z$.
            Then $(xy)z = x(yz) = xz = z$ implying $xy \fleq z$.
    \end{enumerate}
\end{proof}

\subsection{Support map}
\label{ssec:support-map}
Another lattice which will be useful in this article uses the support map which we define next.
Recall that a \defn{principal ideal} of a band $S$ is a two-sided ideal $S^1xS^1$ generated by an element $x \in S$ where $S^1$ is the band $S$ together with an identity if $S$ does not already have one.
In \cite{Clifford1941} Clifford showed that the principal ideals of a band are closed under intersection ($S^1aS^1 \cap S^1bS^1 = S^1abS^1$) and that these principal ideals form a meet-semilattice.
The map which takes $S$ to the set of principal ideals is known as the \defn{support map} and the lattice of principal ideals is known as the \defn{support lattice}.
For the rest of the paper we let
\[
    \supp : S \to (L, \cleq)
\]
denote the support map which sends $S$ to the support lattice.
Since this map sends $S$ to a lattice, then for $x, y \in S$ we have $\supp(xy) = \supp(x) \join \supp(y)$ in the support lattice.
In terms of left regular bands, since $xyx = xy$, the support map can be simplified so that $\supp(y) \cleq \supp(x)$ if and only if $xy = x$
A nice consequence of the support map is that it gives us an equivalent way of viewing chambers.
In particular, an element is a chamber in $S$ if and only if it gets mapped to the maximal element in the support lattice.

The next lemma describes when the product of two facets is a chamber.
\begin{lemma}
    \label{lem:mult-support}
    Let $S$ be a left regular band and let $x$ and $y$ be elements in $S$ and $C$ a chamber.
    If $x \fcover C$ and $y \fcover C$ then $xy = C$ if and only if $\supp(x) \neq \supp(y)$.
\end{lemma}
\begin{proof}
    If $\supp(x) = \supp(y)$ then $\supp(xy) = \supp(x)$ implying $xy = x$ and $yx = y$.
    Since $x \fcover C$ and $y \fcover C$ this implies $C \neq x = xy$ as desired.

    If $\supp(x) \neq \supp(y)$ then $\supp(xy) = \supp(x) \join \supp(y)$ is strictly above $\supp(x)$ in the support lattice.
    But this implies $\supp(xy) = \supp(C')$ for some $C' \in \mcC$.
    In other words $xy \in \mcC$.
    By \autoref{lem:fleq} we know $xy \fleq C$ implying that $xy = C$ as desired.
\end{proof}

\subsection{Example - Hyperplane arrangements}
\label{ssec:hyp-arr}
We next give an important example which is one of the motivating factors for this research.
We give some basic definitions and refer the reader to \cite{TopicsInHypArr} for a more detailed introduction to hyperplane arrangements.

A (central) \defn{hyperplane arrangement} $\msA$ is a finite set of linear hyperplanes in $\mbbR^n$.
Each hyperplane $H \in \msA$ determines two open half-spaces in $\mbbR^n$ which we denote by $H^+$ and $H^-$ arbitrarily.
Note that once a choice on ``positive'' and ``negative'' sides has been made, these choices are fixed.
We let $H^0$ denote the hyperplane itself.

A \defn{face} $X$ of a hyperplane arrangement $\msA$ is then the intersection of some number of hyperplanes and some open half-space of the remaining hyperplanes:
\[
    X = \bigcap_{H \in \msA} H^{\sigma_H(X)} \qquad \text{where }\sigma_H(X) \in \left\{ 0, +, - \right\}.
\]
We let $\sigma(X) = \left( \sigma_H(X) \right)_{H \in \msA}$ denote the \defn{sign sequence} of a face $X$ for some ordering of the hyperplanes in $\msA$.
Let $\msF_\msA$ denote the set of all faces of $\msA$.
The \defn{face poset $\left( \msF_\msA, \fleq \right)$} is the set of faces together with the partial order
\[
    X \fleq Y \text{ if and only if } X \subseteq \ov{Y}.
\]
The set of maximal elements in the face poset are known as \defn{chambers} and is denoted by $\mcC_\msA$.

We next define a left regular band structure on the faces of an arrangement.
For $X, Y \in \msF_\msA$, let $X \circ Y$ be the face whose sign sequence is given by
\[
    \sigma_H(X \circ Y) = \begin{cases}
        \sigma_H(X)& \text{if }\sigma_H(X) \neq 0,\\
        \sigma_H(Y)& \text{otherwise}.
    \end{cases}
\]
Intuitively, what this map does is it takes a generic point in the interior of $X$ and translates the point an infinitesimal distance towards a generic point in the interior of $Y$.
The face $X \circ Y$ is then the face that the generic point of $X$ lands in after this translation.
It turns out that $(\msF_{\msA}, \circ)$ is a left regular band!

It can be verified that the definition of a face poset for a hyperplane arrangement coincides with the definition of a face poset given for left regular bands.
In other words,
\[
    X \fleq Y \iff X \subseteq \ov{Y} \iff X \circ Y = Y.
\]

The support map has an analog in terms of hyperplane arrangements.
For any face $X$ let the support of $X$ be the set of hyperplanes which contain $X$, \ie $\supp(X) = \set{H \in \mcA}{X \subseteq H^0}$.
Ordering these sets by reverse inclusion gives us the support lattice.

\begin{example}
    \label{ex:hyp-arr}
    Let's review these terms with a concrete example.
    Assume we have the hyperplane arrangement $\msA = \left\{ H_1, H_2, H_3 \right\}$.
    This can be seen in the following figure in which the faces have been labeled.
    Here we have $13$ faces including the $6$ chambers ($R_0,\, R_1,\, R_2,\,\ldots,\,R_5$), the $6$ facets ($F_0, \ldots, F_5$) and the center.
    \begin{center}
        {\footnotesize
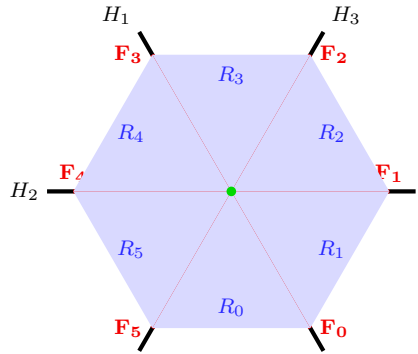
        \begin{tikzpicture}
            [
                scale=0.7,
                fdomain/.style={fill=blue!15!white,color=blue!15!white,opacity=0.3},
                vdomain/.style={color=blue!85!white},
                face/.style={draw=red!95!black,fill=red!95!black,color=red!95!black,ultra thick, -}, 
                norm/.style={black,->},
                vertex/.style={inner sep=1pt,circle,draw=green!85!black,fill=green!85!black,thick},
            ]
            %
            \rs[1][white]{3}{}[30]
            %
            \draw[norm] (0) -- (4) node[below left] {\tiny$e_1$};
            \draw[norm] (0) -- (5) node[below] {\tiny$e_2$};
            \draw[norm] (0) -- (6) node[below right] {\tiny$e_3$};
            %
            \rs[3.5][ultra thick]{3}{}[60]
            \node[above right] at (1) {$H_{3}$};
            \node[above left] at (2) {$H_{1}$};
            \node[left] at (3) {$H_{2}$};
            \rs[3][dashed]{3}{}[240]
            \draw[face] (0) -- (2) node[right] {$\mathbf{F_0}$};
            \draw[face] (0) -- (3) node[above] {$\mathbf{F_1}$};
            \draw[face] (0) -- (4) node[right] {$\mathbf{F_2}$};
            \draw[face] (0) -- (5) node[left] {$\mathbf{F_3}$};
            \draw[face] (0) -- (6) node[above] {$\mathbf{F_4}$};
            \draw[face] (0) -- (1) node[left] {$\mathbf{F_5}$};
            %
            %
            \fill[fdomain] (0) -- (1) -- (2) -- cycle {}; 
            \fill[fdomain] (0) -- (3) -- (2) -- cycle {}; 
            \fill[fdomain] (0) -- (3) -- (4) -- cycle {}; 
            \fill[fdomain] (0) -- (5) -- (4) -- cycle {}; 
            \fill[fdomain] (0) -- (5) -- (6) -- cycle {}; 
            \fill[fdomain] (0) -- (1) -- (6) -- cycle {}; 
            \node[vdomain] at (270:2.2) {$R_0$};
            \node[vdomain] at (330:2.2) {$R_1$};
            \node[vdomain] at (30:2.2) {$R_2$};
            \node[vdomain] at (90:2.2) {$R_3$};
            \node[vdomain] at (150:2.2) {$R_4$};
            \node[vdomain] at (210:2.2) {$R_5$};
            \node[vertex] at (0) {}; 
        \end{tikzpicture}
        }
    \end{center}
    The three vectors $e_1,\,e_2,\,e_3$ are the normal vectors and tell us which side the positive open half-space is for each hyperplane.
    With this information we are able to give sign sequences for faces.
    For example, the sign sequence associated to $F_0$ is $(0,\,+,\,+)$, the sign sequence associated to $F_2$ is $(-,\,-,\,0)$ and the sign sequence associated to $R_4$ is $(+,\,-,\,-)$.
    To understand how $\circ$ works, let us study at $F_2 \circ R_4$.
    In this case we have $(-,\,-,\,0) \circ (+,\,-,\,-) = (-,\,-,\,-)$ which is the sign sequence associated to $R_3$.
    This helps us understand the geometric intuition behind this operation.

    The following figure gives the face poset on the left and the support lattice on the right.
    
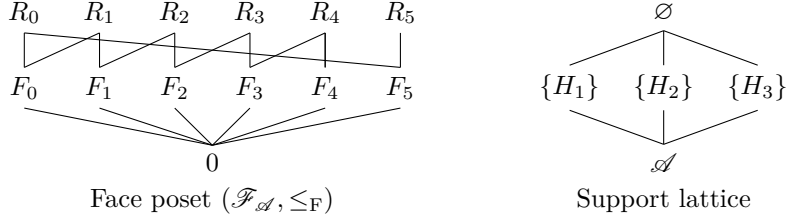
\begin{center}
        \begin{tikzpicture}
            \begin{scope}[shift={(0,0)}]
                \node (0) at (0,0) {$0$};
                \node (F0) at (-2.5, 1) {$F_0$};
                \node (F1) at (-1.5, 1) {$F_1$};
                \node (F2) at (-0.5, 1) {$F_2$};
                \node (F3) at (0.5, 1) {$F_3$};
                \node (F4) at (1.5, 1) {$F_4$};
                \node (F5) at (2.5, 1) {$F_5$};
                \node (R0) at (-2.5, 2) {$R_0$};
                \node (R1) at (-1.5, 2) {$R_1$};
                \node (R2) at (-0.5, 2) {$R_2$};
                \node (R3) at (0.5, 2) {$R_3$};
                \node (R4) at (1.5, 2) {$R_4$};
                \node (R5) at (2.5, 2) {$R_5$};
                \draw (0.north) -- (F0.south);
                \draw (0.north) -- (F1.south);
                \draw (0.north) -- (F2.south);
                \draw (0.north) -- (F3.south);
                \draw (0.north) -- (F4.south);
                \draw (0.north) -- (F5.south);
                \draw (F0.north) -- (R0.south);
                \draw (F0.north) -- (R1.south);
                \draw (F1.north) -- (R1.south);
                \draw (F1.north) -- (R2.south);
                \draw (F2.north) -- (R2.south);
                \draw (F2.north) -- (R3.south);
                \draw (F3.north) -- (R3.south);
                \draw (F3.north) -- (R4.south);
                \draw (F4.north) -- (R4.south);
                \draw (F4.north) -- (R4.south);
                \draw (F5.north) -- (R5.south);
                \draw (F5.north) -- (R0.south);
                \node at (0, -0.5) {Face poset $(\msF_\msA, \fleq)$};
            \end{scope}
            \begin{scope}[shift={(6,0)}]
                \node (0) at (0,0) {$\msA$};
                \node (H3) at (1.25,1) {$\left\{ H_3\right\}$};
                \node (H2) at (0,1) {$\left\{ H_2\right\}$};
                \node (H1) at (-1.25,1) {$\left\{ H_1\right\}$};
                \node (T) at (0, 2) {$\varnothing$};
                \draw (0.north) -- (H1.south);
                \draw (0.north) -- (H2.south);
                \draw (0.north) -- (H3.south);
                \draw (T.south) -- (H1.north);
                \draw (T.south) -- (H2.north);
                \draw (T.south) -- (H3.north);
                \node at (0, -0.5) {Support lattice};
            \end{scope}
        \end{tikzpicture}
    \end{center}
    In this example, the support of $0$ is $\msA$, the support of $F_0$ and $F_3$ is $\left\{ H_1 \right\}$, and the support of every chamber is $\varnothing$.
\end{example}

\section{MC Left Regular Bands}\label{sec:mc-left-regular-bands}
In this section we describe a new family of left regular bands which we call MC left regular bands.
We then give a nice description of its adjacency graph.

\subsection{Adjacency graph}
\label{ssec:adjacency-graph}
We first define a particular graph for left regular bands which will be useful in our study.
We start by defining some basic terms for graphs.
An \defn{edge labeled graph $\Gamma$} is a pair $(V, E)$ where $V$ is a set of vertices and $E$ is a set of triplets $(V_1, V_2, \ell)$ where $V_1$ and $V_2$ are the (vertex) endpoints of the edge and $\ell$ is the label.
We assume our graphs to be \defn{simple}, \ie between every two vertices there is at most one edge, and for every edge $(V_1, V_2, \ell)$ then $V_1 \neq V_2$ (no loops).
A \defn{path} is a sequence of vertices $V_1, \ldots, V_n$ and a \defn{cycle} is a path where $V_1 = V_n$.
The length of a path is the number of edges (or the number of vertices minus $1$).
If a cycle has no repeated vertices (except $V_1 = V_n$) then we say the cycle is \defn{simple}.
Given two edges $e_1$ and $e_2$, then a \defn{minimal cycle containing $e_1$ and $e_2$} is a simple cycle of minimal length which passes through the edges $e_1$ and $e_2$.

Let $S$ be a finite left regular band and let $\mcC_S$ (or $\mcC$) denote the set of chambers for $S$, \ie the set of all maximal elements in the face poset $(S, \fleq)$.
We will define an adjacency graph with a particular edge labeling which will help in future proofs.
For each $x \in S$ which is covered by a chamber of the face poset, recall that $\supp(x)$ is the support of $x$.
The \defn{fiber} of $\supp(x)$ is then the inverse image of the support map and is, in general, a set of multiple elements in $S$ which are all covered by chambers.
Without loss of generality we give some ordering to the elements in the fiber, \ie given the element $\supp(x)$ in the image of the support map for $x$, for each $y \in \supp^{-1}(x)$ we associate a (unique) number $i \in [n]$ where $n = \order{\supp^{-1}(x)}$.
By abuse of notation, we associate each facet $x$ with the notation $x_i$ or $(x, i)$ where $i$ is the unique number given by the above ordering for the fiber of $\supp(x)$.
In other words, if we have the labels $a_1$, $a_2$ and $b_1$ then we know that $\supp(a_1) = \supp(a_2) \neq \supp(b_1)$ and $a_1 \neq a_2$.
We say that two chambers $C$ and $C'$ are \defn{adjacent} if there exists an element $x \in S$ such that $C$ and $C'$ both cover $x$ in the face poset, \ie $x \fcover C$ and $x \fcover C'$.
In this case, we say that $C$ and $C'$ are \defn{adjacent through $x$}.
The \defn{adjacency graph of $S$, $\Gamma_\mcC$},  is the edge labeled graph where the set of vertices are the chambers $\mcC$ and there is an edge with label $x_i$ between two chambers $C$ and $C'$ whenever $C$ and $C'$ are adjacent through $x_i$.
Note that the adjacency graph is not always connected and that two chambers might have multiple edges between them.

\begin{example}
    We continue our hyperplane arrangement example from \autoref{ex:hyp-arr}.
    We have $6$ chambers and so we will have $6$ vertices, each labeled with the appropriate chamber.
    For the facets, we will associate each facet with our new notation.
    This gives us the following list of labels:
    \begin{align*}
        a_1 &=  F_0 &b_1 &=  F_1 &c_1 &=  F_2\\
        a_2 &=  F_3 &b_2 &=  F_4 &c_2 &=  F_5
    \end{align*}
    
    This gives us the following edge labeled adjacency graph
    
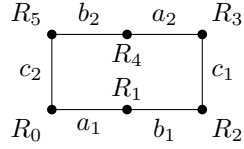
\begin{center}
        \begin{tikzpicture}
            [vertex/.style={inner sep=1pt,circle,draw=black,fill=black,thick}]
                \coordinate (A) at (0,0);
                \coordinate (B) at (1,0);
                \coordinate (C) at (1,1);
                \coordinate (D) at (0,1);
                \coordinate (E) at (2,0);
                \coordinate (F) at (2,1);
                \node[vertex] at (A) {};
                \node[vertex] at (B) {};
                \node[vertex] at (C) {};
                \node[vertex] at (D) {};
                \node[vertex] at (E) {};
                \node[vertex] at (F) {};
                \draw (A) -- (E) -- (F) -- (D) -- (A);
                \node[below left] at (A) {$R_0$};
                \node[above] at (B) {$R_1$};
                \node[below] at (C) {$R_4$};
                \node[above left] at (D) {$R_5$};
                \node[below right] at (E) {$R_2$};
                \node[above right] at (F) {$R_3$};

                \node[left] at (0, 0.5) {$c_2$};
                \node[right] at (2, 0.5) {$c_1$};
                \node[above] at (1.5, 1) {$a_2$};
                \node[above] at (0.5, 1) {$b_2$};
                \node[below] at (0.5, 0) {$a_1$};
                \node[below] at (1.5, 0) {$b_1$};
        \end{tikzpicture}
    \end{center}
\end{example}

\subsection{Example - Free left regular bands}
\label{ssec:example-free-left-regular-bands}
A \defn{free left regular band} is a left regular band $S$ on an alphabet $A$ with no additional relations.
An element in $S$ is any repetition-free word over the alphabet $A$ where repetition-free means that no letter appears more than once.
It can be shown that $\mcC_S$ is the set of maximal words over the alphabet $A$ and that they are just permutations of all the letters in $A$.
Each chamber has a unique facet which is the chamber with the final (right-hand) letter chopped off.
As only one letter can be added to each facet, it can be shown that each facet only has one chamber above it in the face poset.
This implies that no two chambers are adjacent to one another and we get an adjacency graph with $\order{A} !$ vertices and no edges.

\begin{example}
    As an example of a free left regular band, let $A = \left\{ a, b, c \right\}$.
    The free left regular band $S$ on $A$ is then all possible words using the letters of $A$ where no two letters repeat.
    As an example, let's understand why no two letters repeat when multiplying two words.
    \[
        ab \cdot ca = abca = a(bc)a = abc
    \]
    The last equality comes from the fact that $xyx = xy$ in a left regular band for any two elements $x$ and $y$ in $S$.

    The face poset for $S$ is given by the following Hasse diagram.
    
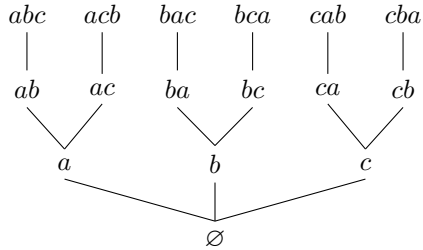
\begin{center}
        \begin{tikzpicture}
            \node (0) at (0,0) {$\varnothing$};

            \node (a) at (-2, 1) {$a$};
            \node (b) at (0, 1) {$b$};
            \node (c) at (2, 1) {$c$};

            \node (ab) at (-2.5, 2) {$ab$};
            \node (ac) at (-1.5, 2) {$ac$};
            \node (ba) at (-0.5, 2) {$ba$};
            \node (bc) at (0.5, 2) {$bc$};
            \node (ca) at (1.5, 2) {$ca$};
            \node (cb) at (2.5, 2) {$cb$};

            \node (abc) at (-2.5, 3) {$abc$};
            \node (acb) at (-1.5, 3) {$acb$};
            \node (bac) at (-0.5, 3) {$bac$};
            \node (bca) at (0.5, 3) {$bca$};
            \node (cab) at (1.5, 3) {$cab$};
            \node (cba) at (2.5, 3) {$cba$};

            \draw (0.north) -- (a.south);
            \draw (0.north) -- (b.south);
            \draw (0.north) -- (c.south);
            \draw (a.north) -- (ab.south);
            \draw (a.north) -- (ac.south);
            \draw (b.north) -- (ba.south);
            \draw (b.north) -- (bc.south);
            \draw (c.north) -- (ca.south);
            \draw (c.north) -- (cb.south);
            \draw (ab.north) -- (abc.south);
            \draw (ba.north) -- (bac.south);
            \draw (ac.north) -- (acb.south);
            \draw (ca.north) -- (cab.south);
            \draw (bc.north) -- (bca.south);
            \draw (cb.north) -- (cba.south);
        \end{tikzpicture}
    \end{center}
    Notice how the number of edges going upwards from any elements is precisely the number of letters which are missing from the word.
    From here, we can see that the adjacency graph of the set of chambers is just the graph on $6$ vertices with no edges.
\end{example}

\subsection{Face meet-semilattice}
An \defn{MC left regular band}  is a left regular band whose face poset is a meet-semilattice and whose adjacency graph is connected.
We will be working with MC left regular bands for the remainder of this article.
The condition of being connected is not very restrictive (in the graph sense) since unconnected components are just the unions of connected ones.
The restriction to having a meet-semilattice face poset is much more restrictive, but will produce ``nice'' adjacency graphs.
In particular, by forcing the meet-semilattice condition, we can show that there is at most one edge between any two vertices.
\begin{lemma}
    \label{lem:one-edge}
    If $(S, \fleq)$ is a meet-semilattice then there exists at most one edge between any two vertices in the adjacency graph $\Gamma_\mcC$.
\end{lemma}
\begin{proof}
    Suppose contrarily that there exists two vertices $C$ and $C'$ which have two (or more) edges labeled $a$ and $b$.
    This implies that $a$ and $b$ are distinct facets in $(S, \fleq)$ which are covered by $C$ and $C'$.
    But this implies that both $a$ and $b$ are maximal lower bounds of $C$ and $C'$ implying that $C$ and $C'$ don't have a meet.
    Therefore $(S, \fleq)$ is not a meet-semilattice.
\end{proof}
For now we will stick with meet-semilattice face posets and leave the question of what happens in a more general setting open.

As we are requiring meet-semilattices, one might then ask when is the face poset a meet-semilattice.
It turns out that there is a way to categorize this.
For ease of notation let $S_{\fgeq x}$ denote the set of elements weakly above $x$ in the face poset, \ie $S_{\fgeq x} = \set{y \in S}{y \fgeq x}$.
\begin{theorem}\label{thm:meet-semilattice}
    The face poset $(S, \fleq)$ is a meet-semilattice if and only if for all $x, y \in S$ one of the following three conditions hold:
    \begin{itemize}
        \item $xy = yx$,
        \item $xy \neq yx$ and $S_{\fgeq x} \cap S_{\fgeq y} = \varnothing$, or
        \item $xy \neq yx$ and there exists a unique $z \in S_{\fgeq x} \cap S_{\fgeq y}$ such that 
            \[
                \supp(z) = \bigmeet_{z' \in S_{\fgeq x} \cap S_{\fgeq y}} \supp(z')
            \]
    \end{itemize}
\end{theorem}
\begin{proof}
    To ease notation we let $\Sigma_{x,y} = S_{\fgeq x} \cap S_{\fgeq y}$.

    \textbf{$\b{\Rightarrow}$}: Suppose that $P = (S, \fleq)$ is a meet-semilattice.
    Let $\hat{P} = (S \cup \left\{ \hat{1} \right\}, \fpleq)$ where $x \fpleq \hat{1}$ for all $x \in S$ and $x \fpleq y$ if and only if $x \fleq y$ for all $x, y \in S$, \ie, $\hat{P}$ is the lattice obtained from $P$ by adjoining a maximal element.
    For $x, y \in S$ we either have that $x \fpjoin y = \hat{1}$ or not.
    
    Suppose first that $x \fpjoin y = \hat{1}$.
    Then $\Sigma_{x,y} = \varnothing$ by definition of join.
    We show that $xy \neq yx$ by contradiction.
    Suppose contrarily that $xy = yx$.
    Recall that by \autoref{lem:fleq} we have $x \fleq xy$ and $y \fleq yx = xy$.
    This implies that either $\Sigma_{x,y}$ is not empty (a contradiction) or that $x = xy = yx = y$ implying that $x = y = \hat{1}$, a contradiction as $\hat{1} \notin S$ by construction.
    Therefore $xy \neq yx$ and $\Sigma_{x,y} = \varnothing$ giving us the second case in the theorem.

    Suppose next that $x \fpjoin y = z \neq \hat{1}$, \ie $z \in S$; then either $xy = z$ or $xy \neq z$.
    If $xy = z$ then since $\supp(xy) = \supp(yx)$ we know that $xy$ and $yx$ have the same rank in $(S, \fleq)$ implying that $yx = z = xy$, putting us in the first case in the theorem.
    If $xy \neq z$, then, since $z$ is the join, $y \not\fleq xy$.
    This implies $yx = yxy \neq xy = xyx$ using both the left regular band property and the fact that $y \not\fleq xy$.
    Therefore $xy \neq yx$ and $x \not\fleq yx$.
    Since $\hat{P}$ is a lattice, then the interval $[z, \hat{1}]$ is also a lattice.
    But then $z$ is the unique minimal element in $\Sigma_{x,y}$ such that $\supp(z) \subseteq \supp(z')$ for all $z' \in \Sigma_{x,y}$.
    In other words, $z$ is the unique minimal element in $\Sigma_{x,y}$ such that $\supp(z) = \bigmeet_{z' \in \Sigma_{x,y}} \supp(z')$, putting us in the third case in the theorem and finishing the first direction.

    \textbf{$\b{\Leftarrow}$}: Suppose that $P = (S, \fleq)$ and $\hat{P}$ is defined as above.
    We aim to show that $\hat{P}$ is a lattice, which will then imply that $P$ is a meet-semilattice.
    Since our poset is finite, it suffices to show that every two elements have a join.
    We do this using our three cases as above.

    Suppose first that $xy = yx$.
    Then, by \autoref{lem:fleq} we have $x \fpleq xy$ and $y \fpleq yx = xy$.
    Suppose that there exists $z \neq xy$ such that $x \fleq z$ and $y \fleq z$.
    (We can assume $z \neq \hat{1}$ since $xy \fpleq \hat{1}$ by definition.)
    By \autoref{lem:fleq} again, we have that $xy \fleq z$.
    This implies that $x \fpjoin y = xy$.

    Suppose next that $xy \neq yx$ and that $\Sigma_{x,y} = \varnothing$.
    This implies that for all $z \in S$ either $x \not\fleq z$ or $y \not\fleq z$.
    Therefore $\hat{1}$ is the only element in $S \cup \left\{ \hat{1} \right\}$ weakly above $x$ and $y$.
    Therefore $x \fpjoin y = \hat{1}$.

    Suppose finally that $xy \neq yx$ and that there is a unique $z$ such that $\supp(z) = \bigmeet_{z' \in \Sigma_{x,y}} \supp(z')$.
    This second condition implies that for all $z' \in \Sigma_{x,y}$ that $\supp(z) \subseteq \supp(z')$.
    Since this $z$ is unique, this implies $z$ is also the unique minimal element in $\Sigma_{x,y}$.
    If not, there exists a second minimal element $z''$ in $\Sigma_{x,y}$, implying $\supp(z) = \supp(z'')$, contradicting the uniqueness of $z$.
    As $z$ is the unique minimal element in $\Sigma_{x,y}$, we have $x \fpjoin y = z$.

    Since every two elements have a join and our poset is finite, therefore $\hat{P}$ is a lattice, implying that $P$ is a meet-semilattice (since removing the top element doesn't affect meets) as desired.
\end{proof}
Although this proof and theorem gives a nice result, it would be interesting to see if there is some simpler condition on a left regular band which is equivalent to its face poset being a meet-semilattice.

\subsection{Properties of the adjacency graph}
To understand this poset more, we explore the adjacency graph first.
It turns out that the adjacency graph is, in general, a very-well behaved graph.
To explain what we mean by a ``very-well behaved graph'' we will first give some preliminary lemmas before giving one of our main theorems.

Before describing our lemmas, we study minimal cycles in the adjacency graph.
Given an arbitrary MC left regular band and two distinct arbitrary edges in the adjacency graph, it is not necessarily the case that there exists a minimal (simple) cycle containing these two edges.
This is shown in the following example:
\begin{example}
    \label{ex:non-cycle}
    Let the following be a the face poset of the MC left regular band which has $n$ chambers, $n-1$ facets and $1$ identity element.
    
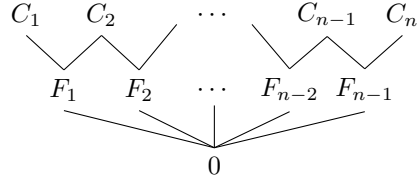
\begin{center}
        \begin{tikzpicture}
                \node (0) at (0,0) {$0$};
                \node (F0) at (-2, 1) {$F_1$};
                \node (F1) at (-1, 1) {$F_2$};
                \node (F2) at (0, 1) {$\cdots$};
                \node (F3) at (1, 1) {$F_{n-2}$};
                \node (F4) at (2, 1) {$F_{n-1}$};
                \node (R0) at (-2.5, 2) {$C_1$};
                \node (R1) at (-1.5, 2) {$C_2$};
                \node (R15) at (-0.5, 2) {};
                \node (R2) at (0, 2) {$\cdots$};
                \node (R25) at (0.5, 2) {};
                \node (R3) at (1.5, 2) {$C_{n-1}$};
                \node (R4) at (2.5, 2) {$C_n$};
                \draw (0.north) -- (F0.south);
                \draw (0.north) -- (F1.south);
                \draw (0.north) -- (F2.south);
                \draw (0.north) -- (F3.south);
                \draw (0.north) -- (F4.south);
                \draw (F0.north) -- (R0.south);
                \draw (F0.north) -- (R1.south);
                \draw (F1.north) -- (R1.south);
                \draw (F1.north) -- (R15.south);
                \draw (F3.north) -- (R25.south);
                \draw (F3.north) -- (R3.south);
                \draw (F4.north) -- (R3.south);
                \draw (F4.north) -- (R4.south);
        \end{tikzpicture}
    \end{center}
    This is a left regular band under the conditions that $R_i X = R_i$, $0 X = X$, $F_i0 = F_i$, and 
    \[
        F_iR_j = \begin{cases}
            R_i& \text{if }j \leq i\\
            R_{i+1} & \text{if } j > i
        \end{cases}
        \qquad
        F_iF_j = \begin{cases}
            R_i& \text{if }j < i\\
            F_i& \text{if }j = i\\
            R_{i + 1}& \text{if } j > i
        \end{cases}
    \]
    The adjacency graph for this left regular band is given by the following graph:
    \begin{center}
        \begin{tikzpicture}
            [vertex/.style={inner sep=1pt,circle,draw=black,fill=black,thick}]
            \coordinate (0) at (0, 0);
            \coordinate (1) at (1, 0);
            \coordinate (2) at (2, 0);
            \coordinate (3) at (3, 0);
            \coordinate (4) at (4, 0);
            \coordinate (5) at (5, 0);
            \node[vertex] at (0) {};
            \node[vertex] at (1) {};
            \node[vertex] at (2) {};
            \node[vertex] at (3) {};
            \node[vertex] at (4) {};
            \node[vertex] at (5) {};
            \draw (0) -- (2);
            \draw (3) -- (5);
            \draw[dotted] (2) -- (3);
        \end{tikzpicture}
    \end{center}
\end{example}

Our first lemma describes what types of minimal cycles exist in the adjacency graph.
\begin{lemma}
    \label{lem:mincyc}
    Let $S$ be a finite MC left regular band with adjacency graph $\Gamma_\mcC$.
    Given two (distinct) edges in $\Gamma_\mcC$ labeled $a$ and $b$, let $\Gamma_{a,b}$ be a minimal cycle in $\Gamma_\mcC$ which contains both edges.
    If $\Gamma_{a, b}$ exists, then
    \begin{itemize}
        \item either no two edges in $\Gamma_{a,b}$ have the same label, or
        \item $\Gamma_{a,b}$ is isomorphic to one of the following two cycles:
            
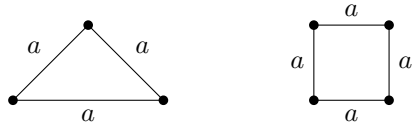
\begin{center}    
                \begin{tikzpicture}
                    [vertex/.style={inner sep=1pt,circle,draw=black,fill=black,thick}]
                    \begin{scope}[shift={(0,0)}]
                        \coordinate (A) at (0,0);
                        \coordinate (B) at (2,0);
                        \coordinate (C) at (1,1);
                        \node[vertex] at (A) {};
                        \node[vertex] at (B) {};
                        \node[vertex] at (C) {};
                        \draw (A) -- (B);
                        \draw (C) -- (B);
                        \draw (A) -- (C);
                        \node[above left] at (0.5, 0.5) {$a$};
                        \node[above right] at (1.5, 0.5) {$a$};
                        \node[below] at (1, 0) {$a$};
                    \end{scope}    
                    \begin{scope}[shift={(4,0)}]
                        \coordinate (A) at (0,0);
                        \coordinate (B) at (1,0);
                        \coordinate (C) at (1,1);
                        \coordinate (D) at (0,1);
                        \node[vertex] at (A) {};
                        \node[vertex] at (B) {};
                        \node[vertex] at (C) {};
                        \node[vertex] at (D) {};
                        \draw (A) -- (B);
                        \draw (C) -- (B);
                        \draw (C) -- (D);
                        \draw (A) -- (D);
                        \node[left] at (0, 0.5) {$a$};
                        \node[right] at (1, 0.5) {$a$};
                        \node[above] at (0.5, 1) {$a$};
                        \node[below] at (0.5, 0) {$a$};
                    \end{scope}
                \end{tikzpicture}
            \end{center}
    \end{itemize}
\end{lemma}
\begin{proof}
    Suppose that there exist two distinct edges in our minimal cycle with the same label and suppose, without loss of generality, that this edge is labeled with $c$.
    Then we have the following
    
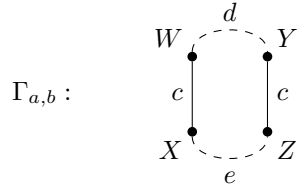
\begin{center}
        \begin{tikzpicture}
            [vertex/.style={inner sep=1pt,circle,draw=black,fill=black,thick}]
            \node at (-2,0) {$\Gamma_{a,b}: $};
            \node[vertex] (W) at (0,0.5) {};
            \node[vertex] (X) at (0,-0.5) {};
            \node[vertex] (Y) at (1,0.5) {};
            \node[vertex] (Z) at (1,-0.5) {};
            \draw (W) -- (X);
            \draw (Y) -- (Z);
            \draw[dashed] (X) to[out=-90, in=-90] (Z);
            \draw[dashed] (W) to[out=90, in=90] (Y);
            \node[left] at (0,0) {$c$};
            \node[right] at (1,0) {$c$};
            \node at (0.5, 1.1) {$d$};
            \node at (0.5, -1.1) {$e$};
            \node[above left] at (W) {$W$};
            \node[above right] at (Y) {$Y$};
            \node[below left] at (X) {$X$};
            \node[below right] at (Z) {$Z$};
        \end{tikzpicture}
    \end{center}
    where $d$ and $e$ are (potentially empty) paths from $W$ to $Y$ and from $X$ to $Z$ respectively.
    Note that $c$ can either be $a$, $b$ or neither.
    Recall that by definition $X, Y, Z, W \in \mcC$ and that $c$ is a facet in $(S, \fleq)$.

    If $W = Y$ and $X=Z$ or if $W = Z$ and $X =Y$, then, as our edges are distinct, this implies there are two edges between the two vertices, contradicting meet-semilattice.
    Therefore, we have $W \neq Y$ or $X \neq Z$ and we have $W \neq Z$ or $X \neq Y$, \ie there is at least one vertex out of $W, X, Y , Z$ that is not equal to either of the two vertices on the other edge.
    Without loss of generality, we may assume that $W \neq Y$ and that $W \neq Z$.
    We make no such restriction on $X$, \ie $X$ could equal either $Y$ or $Z$.
    By definition, $c \fcover W$, $c \fcover Y$ and $c \fcover Z$ since there is an edge connected to $W$, $Y$ and $Z$ with label $c$.
    But since $c$ is covered by all three, that means we should have an edge labeled $c$ between $W$ and $Y$ in addition to between $W$ and $Z$.
    In other words, our cycle looks like:
    
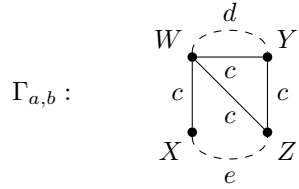
\begin{center}
        \begin{tikzpicture}
            [vertex/.style={inner sep=1pt,circle,draw=black,fill=black,thick}]
            \node at (-2,0) {$\Gamma_{a,b}: $};
            \node[vertex] (W) at (0,0.5) {};
            \node[vertex] (X) at (0,-0.5) {};
            \node[vertex] (Y) at (1,0.5) {};
            \node[vertex] (Z) at (1,-0.5) {};
            \draw (W) -- (X);
            \draw (Y) -- (Z);
            \draw (W) -- (Y);
            \draw (W) -- (Z);
            \draw[dashed] (X) to[out=-90, in=-90] (Z);
            \draw[dashed] (W) to[out=90, in=90] (Y);
            \node[left] at (0,0) {$c$};
            \node[right] at (1,0) {$c$};
            \node[above] at (0.5,-0.5) {$c$};
            \node[below] at (0.5,0.5) {$c$};
            \node at (0.5, 1.1) {$d$};
            \node at (0.5, -1.1) {$e$};
            \node[above left] at (W) {$W$};
            \node[above right] at (Y) {$Y$};
            \node[below left] at (X) {$X$};
            \node[below right] at (Z) {$Z$};
        \end{tikzpicture}
    \end{center}
    Since $W \neq Y$, $d$ is non-empty.
    Furthermore, by minimality of our cycle, $d$ must be of length one.
    Finally, by meet-semilattice, no two vertices can have more than one edge between them, implying that $d$ is nothing more than the edge labeled $c$ and therefore we can remove it from our diagram.

    We now turn to $X$ and continue down a similar path.
    We have three cases.

    \begin{enumerate}
        \item If $X = Z$ then
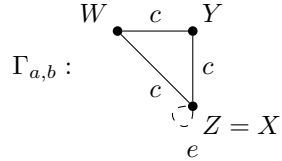
    \begin{center}
        \begin{tikzpicture}
            [vertex/.style={inner sep=1pt,circle,draw=black,fill=black,thick}]
            \node at (-1,0) {$\Gamma_{a,b}: $};
            \node[vertex] (W) at (0,0.5) {};
            \node[vertex] (Y) at (1,0.5) {};
            \node[vertex] (Z) at (1,-0.5) {};
            \draw (Y) -- (Z);
            \draw (W) -- (Y);
            \draw (W) -- (Z);
            \draw[dashed] (Z) to[out=-90, in=180, loop] ();
            \node[right] at (1,0) {$c$};
            \node[above] at (0.5,-0.5) {$c$};
            \node[above] at (0.5,0.5) {$c$};
            \node at (1, -1.1) {$e$};
            \node[above left] at (W) {$W$};
            \node[above right] at (Y) {$Y$};
            \node[below right] at (Z) {$Z = X$};
        \end{tikzpicture}
    \end{center}
        \item If $X = Y$ then 
    
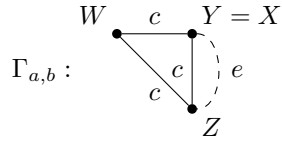
\begin{center}
        \begin{tikzpicture}
            [vertex/.style={inner sep=1pt,circle,draw=black,fill=black,thick}]
            \node at (-1,0) {$\Gamma_{a,b}: $};
            \node[vertex] (W) at (0,0.5) {};
            \node[vertex] (Y) at (1,0.5) {};
            \node[vertex] (Z) at (1,-0.5) {};
            \draw (Y) -- (Z);
            \draw (W) -- (Y);
            \draw (W) -- (Z);
            \draw[dashed] (Y) to[out=0, in=0] (Z);
            \node[left] at (1,0) {$c$};
            \node[above] at (0.5,-0.5) {$c$};
            \node[above] at (0.5,0.5) {$c$};
            \node at (1.6, 0) {$e$};
            \node[above left] at (W) {$W$};
            \node[above right] at (Y) {$Y = X$};
            \node[below right] at (Z) {$Z$};
        \end{tikzpicture}
    \end{center}
        \item If $X \neq Y$ and $X \neq Z$, then by a similar argument as before, there exists an edge labeled with $c$ between $X$ and $Y$ and also between $X$ and $Z$.
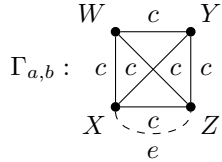
    \begin{center}
        \begin{tikzpicture}
            [vertex/.style={inner sep=1pt,circle,draw=black,fill=black,thick}]
            \node at (-1,0) {$\Gamma_{a,b}: $};
            \node[vertex] (W) at (0,0.5) {};
            \node[vertex] (X) at (0,-0.5) {};
            \node[vertex] (Y) at (1,0.5) {};
            \node[vertex] (Z) at (1,-0.5) {};
            \draw (W) -- (X);
            \draw (Y) -- (Z);
            \draw (W) -- (Y);
            \draw (W) -- (Z);
            \draw (X) -- (Y);
            \draw (X) -- (Z);
            \draw[dashed] (X) to[out=-90, in=-90] (Z);
            \node[left] at (0,0) {$c$};
            \node[right] at (1,0) {$c$};
            \node[right] at (0,0) {$c$};
            \node[left] at (1,0) {$c$};
            \node[below] at (0.5,-0.5) {$c$};
            \node[above] at (0.5,0.5) {$c$};
            \node at (0.5, -1.1) {$e$};
            \node[above left] at (W) {$W$};
            \node[above right] at (Y) {$Y$};
            \node[below left] at (X) {$X$};
            \node[below right] at (Z) {$Z$};
        \end{tikzpicture}
    \end{center}
    \end{enumerate}
    In all three cases, by the minimality of $\Gamma_{a,b}$ and by the meet-semilattice condition, the edge $e$ isn't needed in our diagrams.
    In particular, this implies that $a = b = c$ and that our minimal cycle is one of the following two cycles:
    
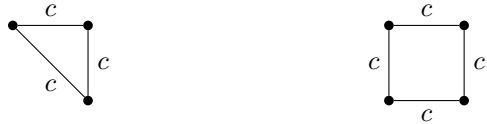
\begin{center}
        \begin{tikzpicture}
            [vertex/.style={inner sep=1pt,circle,draw=black,fill=black,thick}]
            \begin{scope}[shift={(0,0)}]
                \node[vertex] (W) at (0,0.5) {};
                \node[vertex] (Y) at (1,0.5) {};
                \node[vertex] (Z) at (1,-0.5) {};
                \draw (Y) -- (Z);
                \draw (W) -- (Y);
                \draw (W) -- (Z);
                \node[right] at (1,0) {$c$};
                \node[above] at (0.5,-0.5) {$c$};
                \node[above] at (0.5,0.5) {$c$};
            \end{scope}
            \begin{scope}[shift={(5,0)}]
                \node[vertex] (W) at (0,0.5) {};
                \node[vertex] (X) at (0,-0.5) {};
                \node[vertex] (Y) at (1,0.5) {};
                \node[vertex] (Z) at (1,-0.5) {};
                \draw (W) -- (X);
                \draw (Y) -- (Z);
                \draw (W) -- (Y);
                \draw (X) -- (Z);
                \node[left] at (0,0) {$c$};
                \node[right] at (1,0) {$c$};
                \node[below] at (0.5,-0.5) {$c$};
                \node[above] at (0.5,0.5) {$c$};
            \end{scope}
        \end{tikzpicture}
    \end{center}
\end{proof}

Inside the proof of the previous lemma, we notice that even in the square cycle case, we actually have the complete graph.
This provides us with a nice corollary to know when a label appears exactly once in our graph.
\begin{corollary}
    \label{cor:no-cycle-implies-exactly-once}
    Let $S$ be a finite MC left regular band with adjacency graph $\Gamma_\mcC$.
    Let $a$ be a (fixed) label of some edge.
    If there are no cycles of length $3$ which contain an edge with label $a$, then the label $a$ appears exactly once in $\Gamma_\mcC$.    
\end{corollary}
\begin{proof}
    Suppose there are two edges with the same label $a$.
    Then there are at least $3$ chambers above $a$ in the face poset.
    These three chambers are therefore connected in the adjacency graph and form a cycle of length $3$ containing $a$ as an edge label.
\end{proof}

We continue down this road and look deeper into what happens when we take two arbitrary edges with different labels.
First, we describe when their supports are different.
\begin{lemma}
    \label{lem:diff-support}
    Let $S$ be a finite MC left regular band with adjacency graph $\Gamma_\mcC$.
    Let $a$ and $b$ be distinct labels of two edges.
    Suppose further that $\supp(a) \neq \supp(b)$.
    Then $ab = aX$ for all chambers $X$ weakly above $b$.
\end{lemma}
\begin{proof}
    Let $X \in \mcC \cap S_{\fgeq b}$ be a chamber weakly above $b$.
    Since $\supp(a) \neq \supp(b)$ this implies $ab \in \mcC$.
    By \autoref{lem:equiv-chamber} we have $ab = (ab)X$.
    By associativity and definition of the face poset we then have $(ab)X = a(bX) = aX$.
    In other words $ab = aX$ for any arbitrary chamber weakly above $b$.
\end{proof}

If $a$ and $b$ share a vertex in common, then we get something a little stronger.
\begin{corollary}
    \label{cor:diff-support}
    Let $S$ be a finite MC left regular band with adjacency graph $\Gamma_\mcC$.
    Let $a$ and $b$ be labels of two edges with vertices $X,\,Y$ and $Y,\, W$ respectively.
    Suppose further that $\supp(a) \neq \supp(b)$.
    Then $ab = ba (= aW = aY = bY = bX = Y)$.
\end{corollary}
\begin{proof}
    This is a direct result of the previous lemma combined with the definition of face poset.
    By \autoref{lem:diff-support} we have $ba = bX = bY$ and $ab = aW = aY$.
    By definition of the face poset we have $aY = Y = bY$.
    Combining these two results we get $ab = ba$ as desired.
\end{proof}

The case for when the supports are the same is a little more complicated, but still describable.
Recall that our edges are labeled by an ordering on the fibers of the support map, \ie $a_1, a_2, b_1 \in S$ implies $\supp(a_1) = \supp(a_2) \neq \supp(b_1)$ and $a_1 \neq a_2$.
\begin{lemma}
    \label{lem:same-support-bij}
    Let $S$ be a finite MC left regular band with adjacency graph $\Gamma_\mcC$.
    Let $a_1$ and $a_2$ be labels of two edges.
    Suppose further that $\supp(a_1) = \supp(a_2)$.
    There exists a bijection
    \[
        \alpha: \mcC \cap S_{\fgeq a_1} \to \mcC \cap S_{\fgeq a_2}
    \]
    where for $X \in \mcC \cap S_{\fgeq a_1}$ there exists a unique $X' \in \mcC \cap S_{\fgeq a_2}$ such that $\alpha(X) = X' = a_2X$ and $\alpha^{-1}(X') = X = a_1X'$.
\end{lemma}
\begin{proof}
    Let $a_1$ and $a_2$ be defined as above.
    As $\supp(a_1) = \supp(a_2)$ then $a_1a_2 = a_1$ and $a_2a_1 = a_2$.
    Let $X \in \mcC \cap S_{\fgeq a_1}$.
    Since $a_1 a_2 = a_1$ and since $a_1 \fleq X$ we have $X = a_1 X = a_1 a_2 X$.
    Since $a_2 \fleq a_2 X \in \mcC$ this implies there exists $X' \in \mcC \cap S_{\fgeq a_2}$ such that $a_2 X = X'$.
    Therefore $a_1 a_2 X = a_1 X'$, giving us $X = a_1 X'$.
    This $X'$ is unique for if not then $a_1 X' = X = a_1 X''$ implies $a_2 a_1 X' = a_2 a_1 X''$ and $a_2 X' = a_2 X''$ (as $a_2 a_1 = a_2$).
    Finally since $X', X'' \in \mcC \cap S_{\fgeq a_2}$ this implies $X' = X''$ as desired.

    The reverse bijection is similar to above with $1$ and $2$ swapped.
\end{proof}

Using the previous lemma we can get a better understanding of what happens with two arbitrary edges if a facet has exactly two chambers weakly above it.
\begin{lemma}
    \label{lem:same-support}
    Let $S$ be a finite MC left regular band with adjacency graph $\Gamma_\mcC$.
    Let $a_1$ and $a_2$ be labels of two edges with vertices $Z,\,W$ and $X,\, Y$ respectively.
    Suppose further that $\supp(a_1) = \supp(a_2)$ and that $\order{\mcC \cap S_{\fgeq a_1}} = 2$.
    Then either
    \begin{itemize}
        \item $a_1X = W$, $a_1Y = Z$, $a_2W = X$, $a_2Z = Y$, or
        \item $a_1X = Z$, $a_1Y = W$, $a_2W = Y$, $a_2Z = X$.
    \end{itemize}
\end{lemma}
\begin{proof}
    By \autoref{lem:same-support-bij}, since $\order{\mcC \cap S_{\fgeq a_1}} = 2$ we know that $\order{\mcC \cap S_{\fgeq a_2}} = 2$ as well.
    Therefore there are $4$ different ways our bijections can go considering that we have two potential options for each mapping:
    \begin{align*}
        a_1 X &= W &or& &a_1 X &=  Z &;& & a_2 W &= X &or & &a_2 W &=  Y\\
        a_1 Y &= Z & & &a_1 Y &=  W & & &a_2 Z &= Y & & &a_2 Z &=  X
    \end{align*}
    In other words, our four options come from choosing one of the two maps on the left and one of the two maps on the right.
    We aim to show that taking the middle two maps or taking the outer two maps give us contradictions.
    Recall that $\supp(a_1) = \supp(a_2)$ implies that $a_1 a_2 = a_1$ and $a_2 a_1 = a_2$.

    Suppose first that we take the inner two maps, \ie $a_1X = Z$, $a_1Y = W$, $a_2 W = X$ and $a_2 Z = Y$.
    Then by associativity and our equalities we have
    \[
        W = a_1Y = a_1(a_2Z) = (a_1 a_2)Z = a_1Z = Z
    \]
    contradicting the fact that $W$ and $Z$ are different chambers.
    Similarly, suppose we take our outer two maps, \ie $a_1X = W$, $a_1Y = Z$, $a_2W = Y$ and $a_2Z = X$.
    By a similar approach we have
    \[
        W = a_1 X = a_1 (a_2 Z) = (a_1 a_2) Z = a_1Z = Z
    \]
    giving us a contradiction.

    To show that taking the 1st and 3rd mappings (or similarly the 2nd and 4th mappings) both give us valid options, it suffices to show associativity.
    Without loss of generality we show that $a_1(a_2 Z) = (a_1 a_2) Z$ for the 1st and 3rd mappings.
    All other combinations are done similarly.
    Our 1st and 3rd mappings imply $a_1X = W$, $a_1Y = Z$, $a_2W = X$ and $a_2Z = Y$.
    \begin{gather*}
        a_1(a_2Z) = a_1 Y = Z\\
        (a_1 a_2) Z = a_1 Z = Z
    \end{gather*}
    as desired.
\end{proof}

Combining the previous few lemmas, we end up with a nice description of cycles in the adjacency graph.
\begin{lemma}
    \label{lem:cycles}
    Let $S$ be a finite MC left regular band with adjacency graph $\Gamma_\mcC$.
    Let $a_1$ be the label of some edge such that $\order{\mcC \cap S_{\fgeq a_1}} = 2$.
    Then every simple cycle $\Gamma$ which contains the edge labeled $a_1$ has an even number of edges $a_i$ such that $\supp(a_i) = \supp(a_1)$.
\end{lemma}
\begin{proof}
    Let $\Gamma$ be a simple cycle which contains $a_1$.
    We label $\Gamma$ in the following way:
    
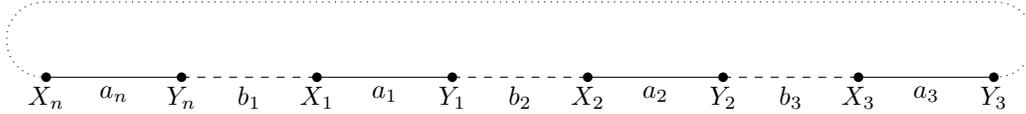
\begin{center}
        \begin{tikzpicture}
            [vertex/.style={inner sep=1pt,circle,draw=black,fill=black,thick}, xscale=1.8]
            \coordinate (xn) at (0,0);
            \coordinate (yn) at (1,0);
            \coordinate (x1) at (2,0);
            \coordinate (y1) at (3,0);
            \coordinate (x2) at (4,0);
            \coordinate (y2) at (5,0);
            \coordinate (x3) at (6,0);
            \coordinate (y3) at (7,0);
            \node[vertex] at (xn) {};
            \node[vertex] at (yn) {};
            \node[vertex] at (x1) {};
            \node[vertex] at (y1) {};
            \node[vertex] at (x2) {};
            \node[vertex] at (y2) {};
            \node[vertex] at (x3) {};
            \node[vertex] at (y3) {};
            \draw (xn) -- (yn);
            \draw (x1) -- (y1);
            \draw (x2) -- (y2);
            \draw (x3) -- (y3);
            \draw[dashed] (yn) -- (x1);
            \draw[dashed] (y1) -- (x2);
            \draw[dashed] (y2) -- (x3);
            \draw[dotted] (y3) to[out=0, in=0] (7,1) -- (0,1) to[out=180, in=180] (xn);
            \node[below] at (xn) {$X_n$};
            \node[below] at (yn) {$Y_n$};
            \node[below] at (x1) {$X_1$};
            \node[below] at (y1) {$Y_1$};
            \node[below] at (x2) {$X_2$};
            \node[below] at (y2) {$Y_2$};
            \node[below] at (x3) {$X_3$};
            \node[below] at (y3) {$Y_3$};
            \node[below] at (0.5, 0) {$a_n$};
            \node[below] at (1.5, 0) {$b_1$};
            \node[below] at (2.5, 0) {$a_1$};
            \node[below] at (3.5, 0) {$b_2$};
            \node[below] at (4.5, 0) {$a_2$};
            \node[below] at (5.5, 0) {$b_3$};
            \node[below] at (6.5, 0) {$a_3$};
        \end{tikzpicture}
    \end{center}
    In this graph each $a_i$ is such that $\supp(a_i) = \supp(a_1)$.
    Each $b_i$ is a chain of edges such that for any edge labeled $x$ in $b_i$ then $\supp(a_1) \neq \supp(x)$.
    Note that $n$ could equal $1$.

    By \autoref{lem:diff-support} and \autoref{cor:diff-support} for every edge labeled $e_i$ in $b_i$ and for every vertex $V_i$ in $b_i$ we have $a_{i-1}e_i = a_{i-1}V_i = Y_{i-1}$ for all $i$ where $a_0 = a_n$.
    Similarly we have $a_{i}e_i = a_i V_i = X_i$ for all $i$.
    In particular, we have $a_{i-1}X_i = Y_{i-1}$ and $a_i Y_{i-1} = X_i$.
    Using these equalities with \autoref{lem:same-support}, combined with the fact that $\order{\mcC \cap S_{\fgeq a_1}} = 2$ , we have that $a_{i-1}Y_i = X_{i-1}$ and $a_{i}X_{i-1} = Y_i$.

    Plugging in $i = 1$ we have $a_1 Y_2 = X_1$ and $a_1 X_2 = Y_1$.
    Using the above for $i = 2$ and the fact that $a_1 a_2 = a_1$ since $\supp(a_1) = \supp(a_2)$ we have
    \begin{gather*}
        a_1 Y_3 = a_1 a_2 X_2 = a_1 X_2 = Y_1\\
        a_1 X_3 = a_1 a_2 Y_2 = a_1 Y_2 = X_1
    \end{gather*}
    Going up inductively, we have that
    \[
        a_1 Y_i = \begin{cases}
            Y_1& \text{if $i$ is odd}\\
            X_1& \text{otherwise}
        \end{cases}
        \quad
        a_1 X_i = \begin{cases}
            X_1& \text{if $i$ is odd}\\
            Y_1& \text{otherwise}
        \end{cases}
    \]
    Since cyclically $a_0$ and $a_n$ are identified, we know that $a_1X_n = Y_1$.
    In other words, $n$ must be even as desired.
\end{proof}

MC left regular bands which have exactly two chambers above every facet are known as \defn{thin} MC left regular bands.
In fact, thin MC left regular bands are precisely the MC left regular band where every facet appears as an edge, and there are no minimal cycles which are triangles.

~\\

\begin{lemma}
    \label{lem:thin_mc_lrb}
    Let $S$ be a finite MC left regular band with adjacency graph $\Gamma_\mcC$.
    Every element covered by a chamber appears as an edge in $\Gamma_\mcC$ and there are no triangles in $\Gamma_\mcC$ if and only if $S$ is thin.
\end{lemma}
\begin{proof}
    Suppose that every element covered by a chamber appears as an edge in $\Gamma_\mcC$ and there are no triangles in $\Gamma_\mcC$.
    Since every element covered by a chamber appears as an edge, then $\order{\mcC \cap S_{\fgeq a}} \geq 2$ for all edge labels $a$.
    Furthermore, as there are no triangles, by \autoref{cor:no-cycle-implies-exactly-once}, every edge label appears exactly once.
    This further implies that $\order{\mcC \cap S_{\fgeq a}} = 2$ for all edge labels $a$.
    Therefore $S$ is thin.

    The reverse direction is clear.
\end{proof}

We can use the previous two lemmas to give us a nice graphical description for any thin MC left regular band.
\begin{theorem}
    \label{cor:to-graph}
    Let $S$ be a finite thin MC left regular band.
    Then the adjacency graph $\Gamma_\mcC$ is an edge-labeled graph where every simple cycle has an even number of edges labeled by elements with the same support and each label appears exactly once.
\end{theorem}
\begin{proof}
    By definition of thin, for every facet $a$, $\order{\mcC \cap S_{\fgeq a}} = 2$.
    Therefore by \autoref{lem:cycles}, every simple cycle has an even number of edges labeled by elements with the same support as desired.
\end{proof}

\hugeskip

\section{Thin left regular band graphs}\label{sec:thin-left-regular-band-graphs}
Motivated by the nice description of graphs coming from thin MC left regular bands, we reverse this process by describing a type of graph which will generate a (finite) thin MC left regular band.
Let $G$ be a (simple) graph $(V, E)$ such that $V = \set{A_i}{i \in I}$ is the set of vertices for some integer interval $I$ and $E$ is the set of labeled edges:
\[
    E = \set{(A_i, A_j, (a,b))}{(a,b) \in \NN^2 \text{ and } A_i, A_j \in V,\, i \neq j}
\]
where $(a,b)$ is the label of the edge between the vertices $A_i$ and $A_j$ and such that no two edges have the same label.
Without loss of generality, we take $I$ to be the set $[n]$ and, as the labels $(a,b)$ are unique, we let $(a,b)$ represent both the label and the edge itself.
Given a path $\Gamma$ in $G$ we let $\Gamma^a$ be the number of edges in $\Gamma$ whose first component is $a$, \ie $\Gamma^a = \order{\set{(a, i) \in \Gamma}{i \in \NN}}$.
Suppose further that $G$ has the property that for each simple cycle $\Gamma$ in $G$ then $\Gamma^a$ is even for all $a \in \NN$.
We call such a graph $G$ with the above properties a \defn{thin left regular band graph}, or more simply a \defn{thin LRB graph}.

\newcommand\cpath[3]{ {{}_{#1}{\Gamma}_{#2}^{#3}} }
Given any thin LRB graph $G$, we can define a thin MC left regular band using the following conventions.
For ease of notation, given two vertices $A$ and $B$, let $\cpath{A}{B}{~}$ denote a minimal length path in $G$ from $A$ to $B$.
Then $\cpath{A}{B}{a}$ denotes the number of edges in $\cpath{A}{B}{~}$ whose first component is $a$.
For all edges $(a_1, b_1)$, $(a_2, b_2)$, $(a_3, b_3)$ and all vertices $A_i$, $A_j$, then we define multiplication in the following way:
\begin{enumerate}
    \item\label{enum:mult-1} $A_i \cdot x = A_i$ for all $x \in V \cup E$.
    \item\label{enum:mult-2} For an edge $(A_i, A_j, (a_1, b_1))$ we have $(a_1, b_1) \cdot A_i = A_i$ and $(a_1, b_1) \cdot A_j = A_j$.
    \item\label{enum:mult-3} If $a_1 = a_2$ then $(a_1, b_1) \cdot (a_2, b_2) = (a_1, b_1)$ and $(a_2, b_2) \cdot (a_1, b_1) = (a_2, b_2)$.
    \item\label{enum:mult-4} For two edges $(A_{i_1}, A_{j_i}, (a_1, b_1))$ and $(A_{i_2}, A_{j_2}, (a_2, b_2))$.
        If $a_1 \neq a_2$, then
        \[
            (a_1, b_1) \cdot A_{i_2} = (a_1, b_1) \cdot A_{j_2} = (a_1, b_1) \cdot (a_2, b_2) = \begin{cases}
                A_{i_1}& \text{if } \order{\cpath{A_{j_1}}{A_{i_2}}{a_1}} \text{ is odd }\\
                A_{j_1}& \text{otherwise}
            \end{cases}
        \]
        If $a_1 = a_2$, then
        \[
            (a_1, b_1) \cdot A_{i_2} = \begin{cases}
                A_{i_1}& \text{if } \order{\cpath{A_{j_1}}{A_{i_2}}{a_1}} \text{ is odd }\\
                A_{j_1}& \text{otherwise}
            \end{cases}
            \quad
            (a_1, b_1) \cdot A_{j_2} = \begin{cases}
                A_{j_1}& \text{if } \order{\cpath{A_{j_1}}{A_{i_2}}{a_1}} \text{ is odd }\\
                A_{i_1}& \text{otherwise}
            \end{cases}
        \]
\end{enumerate}

It turns out a thin LRB graph with the above multiplication is a thin MC left regular band.

{
\begin{theorem}
    Let $G = (V, E)$ be a thin LRB graph as defined above.
    Then $(V \cup E \cup \left\{ 0 \right\}, \cdot)$ is a thin MC left regular band where multiplication is defined as above with the addition that $0$ is the identity element, \ie $x \cdot 0 = 0 \cdot x = x$ for all $x \in V \cup E \cup \left\{ 0 \right\}$.
\end{theorem}
\begin{proof}
    If $(V \cup E, \cdot)$ is a left regular band, it's clear that it must be thin as every edge appears once and has two distinct vertices attached to it, and is connected as $G$ is connected.
    It is a meet-semilattice by the addition of the identity element.

    To show that $(V \cup E, \cdot)$ is a left regular band, we must show it is idempotent, associative and that $xyx = xy$ for all $x, y \in V \cup E$.

    \textbf{Idempotent}: If $x \in E$ then by \autoref{enum:mult-3} we have idempotency. If $x \in V$ then by \autoref{enum:mult-1} we have idempotency.

    \textbf{Associativity}: We aim to show $(xy)z = x(yz)$ for all $x, y, z$.
    We break this down into different cases depending on whether our elements are edges or vertices.

    \begin{itemize}
        \item If $x = A_i$ is a vertex, then by repeated applications of \autoref{enum:mult-1} we have
            \[
                (A_i \cdot y) \cdot z = A_i \cdot z = A_i = A_i \cdot (y \cdot z)
            \]
        \item If $x = (a_1, b_1)$ is an edge and $y = A_i$ is a vertex then
            \[
                (x \cdot A_i) \cdot z = A_j \cdot z = A_j = x\cdot A_i = x \cdot (A_i \cdot z)
            \]
            where the first and third equality ($x \cdot A_i = A_j$) is coming from \autoref{enum:mult-2} or \autoref{enum:mult-4} depending on $x$, the second and fourth equalities are coming from \autoref{enum:mult-1}.
        \item Let $x = (A_{i_1}, A_{j_1}, (a_1, b_1))$, $y = (A_{i_2}, A_{j_2}, (a_2, b_2))$, $z = (A_{i_3}, A_{j_3}, (a_3, b_3))$ and $z' = A_{i_3}$ (without loss of generality).
            We show $(xy)z = x(yz)$ and $(xy)z' = x(yz')$ at the same time.
            This case is dependent on the $a_i$.

            \textbf{$\b{a_1 = a_2 = a_3}$}: then
            \begin{align*}
                \left( (a_1, b_1) \cdot (a_2, b_2) \right) \cdot (a_3, b_3) &\xequal{\autoref{enum:mult-3}}  (a_1, b_1) \cdot (a_3, b_3)\\
                &\xequal{\autoref{enum:mult-3}}  (a_1, b_1) \\
                &\xequal{\autoref{enum:mult-3}}  (a_1, b_1) \cdot (a_2, b_2) \\
                &\xequal{\autoref{enum:mult-3}}  (a_1, b_1) \cdot \left( (a_2, b_2) \cdot (a_3, b_3) \right)
            \end{align*}
            for $z$.
            For the $z'$ case we have
            \[
                \left( (a_1, b_1) \cdot (a_2, b_2) \right) \cdot A_{i_3} 
                \xequal{\autoref{enum:mult-3}}  (a_1, b_1) \cdot A_{i_3} 
                \xequal{\autoref{enum:mult-4}}  \begin{cases}
                    A_{i_1}& \text{if } \order{\cpath{A_{j_1}}{A_{i_3}}{a_1}} \text{ is odd }\\
                    A_{j_1}& \text{otherwise}.
                \end{cases}
            \]
            Likewise
            \begin{align*}
                (a_1, b_1) \cdot \left( (a_2, b_2) \cdot A_{i_3} \right) 
                &\xequal{\autoref{enum:mult-4}} \begin{cases}
                    (a_1, b_1) \cdot A_{i_2}& \text{if } \order{\cpath{A_{j_2}}{A_{i_3}}{a_2}} \text{ is odd }\\
                    (a_1, b_1) \cdot A_{j_2}& \text{otherwise}
                \end{cases}\\
                &\xequal{\autoref{enum:mult-4}}  \begin{cases}
                    A_{i_1}& \text{if } \order{\cpath{A_{j_2}}{A_{i_3}}{a_2}} + \order{\cpath{A_{j_1}}{A_{i_2}}{a_1}} \text{ is even }\\
                    A_{j_1}& \text{otherwise}.
                \end{cases}
            \end{align*}
            But notice that the path from $A_{j_1}$ to $A_{i_3}$ in the second case contains an additional $a_1 \; (= a_2)$ as it traverses the edge $(a_2, b_2)$.
            The two paths defined are either the same path (and therefore equal) or they create a cycle which (by definition) contains an even number of edges labeled $(a_1, -)$.
            In both cases we have the desired equality.

            \textbf{$\b{a_1 \neq a_2 \neq a_3}$}: then
            \begin{align*}
                \left( (a_1, b_1) \cdot (a_2, b_2) \right) \cdot A_{i_3} &\xequal{\autoref{enum:mult-4}}  \left( (a_1, b_1) \cdot (a_2, b_2) \right) \cdot \left( a_3, b_3 \right) \\
                &\xequal{\autoref{enum:mult-4}}  \begin{cases}
                    A_{i_1} \cdot (a_3, b_3)& \text{if }\order{\cpath{A_{j_1}}{A_{i_2}}{a_1}} \text{ is odd}\\
                    A_{j_1} \cdot (a_3, b_3)& \text{otherwise}
                \end{cases}\\
                &\xequal{\autoref{enum:mult-1}}  \begin{cases}
                    A_{i_1} & \text{if }\order{\cpath{A_{j_1}}{A_{i_2}}{a_1}} \text{ is odd}\\
                    A_{j_1} & \text{otherwise}
                \end{cases}\\
                &\xequal{\autoref{enum:mult-4}}  (a_1, b_1) \cdot (a_2, b_2)
            \end{align*}
            Similarly,
            \begin{align*}
                (a_1, b_1) \cdot \left( (a_2, b_2) \cdot A_{i_3} \right) &\xequal{\autoref{enum:mult-4}}  (a_1, b_1) \cdot \left( (a_2, b_2) \cdot (a_3, b_3) \right)\\
                &\xequal{\autoref{enum:mult-4}}  \begin{cases}
                    (a_1, b_1) \cdot A_{i_2}& \text{if }\order{\cpath{A_{j_2}}{A_{i_3}}{a_2}} \text{ is odd}\\
                    (a_1, b_1) \cdot A_{j_2}& \text{ otherwise}
                \end{cases}\\
                &\xequal{\autoref{enum:mult-4}}  (a_1, b_1) \cdot (a_2, b_2)
            \end{align*}
            as desired.

            \textbf{$\b{a_2 = a_3 \neq a_1}$}: Then
            \begin{align*}
                \left( (a_1, b_1) \cdot (a_2, b_2) \right) \cdot (a_3, b_3) &\xequal{\autoref{enum:mult-4}}  \begin{cases}
                    A_{i_1} \cdot (a_3, b_3) & \text{if }\order{\cpath{A_{j_1}}{A_{i_2}}{a_1}} \text{ is odd}\\
                    A_{j_1} \cdot (a_3, b_3) & \text{ otherwise}
                \end{cases}\\
                &\xequal{\autoref{enum:mult-1}}  \begin{cases}
                    A_{i_1} & \text{if }\order{\cpath{A_{j_1}}{A_{i_2}}{a_1}} \text{ is odd}\\
                    A_{j_1} & \text{ otherwise}
                \end{cases}\\
                &\xequal{\autoref{enum:mult-4}}  (a_1, b_1) \cdot (a_2, b_2).
            \end{align*}
            Similarly, $\left( (a_1, b_1) \cdot (a_2, b_2) \right) \cdot A_{i_3} = (a_1, b_1) \cdot (a_2, b_2)$.
            On the other hand, by \autoref{enum:mult-3} we have $(a_1, b_1) \cdot \left( (a_2, b_2) \cdot (a_3, b_3) \right) = (a_1, b_1) \cdot (a_2, b_2)$
            and
            \begin{align*}
                (a_1, b_1) \cdot \left( (a_2, b_2) \cdot A_{i_3} \right) &\xequal{\autoref{enum:mult-4}}  \begin{cases}
                    (a_1, b_1) \cdot A_{i_2}& \text{if } \order{\cpath{A_{j_2}}{A_{i_3}}{a_2}} \text{ is odd}\\
                    (a_1, b_1) \cdot A_{j_2}& \text{otherwise}
                \end{cases}\\
                &\xequal{\autoref{enum:mult-4}}  \begin{cases}
                    A_{i_1}& \text{if } \order{\cpath{A_{j_1}}{A_{i_2}}{a_2}} \text{ is odd}\\
                    A_{j_1}& \text{otherwise}
                \end{cases}\\
                &\xequal{\autoref{enum:mult-4}} (a_1, b_1) \cdot (a_2, b_2)
            \end{align*}
            as desired.

            \textbf{$\b{a_1 = a_3 \neq a_2}$}:
            Then
            \begin{align*}
                \left( (a_1, b_1) \cdot (a_2, b_2) \right) \cdot (a_3, b_3) &\xequal{\autoref{enum:mult-4}}  \begin{cases}
                    A_{i_1} \cdot (a_3, b_3) & \text{if }\order{\cpath{A_{j_1}}{A_{i_2}}{a_1}} \text{ is odd}\\
                    A_{j_1} \cdot (a_3, b_3) & \text{ otherwise}
                \end{cases}\\
                &\xequal{\autoref{enum:mult-1}}  \begin{cases}
                    A_{i_1} & \text{if }\order{\cpath{A_{j_1}}{A_{i_2}}{a_1}} \text{ is odd}\\
                    A_{j_1} & \text{ otherwise}
                \end{cases}\\
                &\xequal{\autoref{enum:mult-4}} (a_1, b_1) \cdot (a_2, b_2)
            \end{align*}
            Similarly, $\left( (a_1, b_1) \cdot (a_2, b_2) \right) \cdot A_{i_3} = (a_1, b_1) \cdot (a_2, b_2)$.
            On the other hand,
            \begin{align*}
                (a_1, b_1) \cdot \left( (a_2, b_2) \cdot A_{i_3} \right) &\xequal{\autoref{enum:mult-4}}  (a_1, b_1) \cdot \left( (a_2, b_2) \cdot (a_3, b_3) \right) \\
                &\xequal{\autoref{enum:mult-4}}  \begin{cases}
                    (a_1, b_1) \cdot A_{i_2}& \text{if } \order{\cpath{A_{j_2}}{A_{i_3}}{a_2}} \text{ is odd}\\
                    (a_1, b_1) \cdot A_{j_2}& \text{otherwise}
                \end{cases}\\
                &\xequal{\autoref{enum:mult-4}} (a_1, b_1) \cdot (a_2, b_2)
            \end{align*}
            as desired.

            \textbf{$\b{a_1 = a_2 \neq a_3}$}: Then
            \begin{align*}
                \left( (a_1, b_1) \cdot (a_2, b_2) \right) \cdot (a_3, b_3) &\xequal{\autoref{enum:mult-3}}  (a_1, b_1) \cdot (a_3, b_3) \\
                &\xequal{\autoref{enum:mult-4}}  (a_1, b_1) \cdot A_{i_3}\\
                &\xequal{\autoref{enum:mult-4}}  \begin{cases}
                    A_{i_1} & \text{if }\order{\cpath{A_{j_1}}{A_{i_2}}{a_1}} \text{ is odd}\\
                    A_{j_1} & \text{ otherwise}
                \end{cases}
            \end{align*}
            On the other hand,
            \begin{align*}
                (a_1, b_1) \cdot \left( (a_2, b_2) \cdot A_{i_3} \right) &\xequal{\autoref{enum:mult-4}}  (a_1, b_1) \cdot \left( (a_2, b_2) \cdot (a_3, b_3) \right) \\
                &\xequal{\autoref{enum:mult-4}}  \begin{cases}
                    (a_1, b_1) \cdot A_{i_2}& \text{if } \order{\cpath{A_{j_2}}{A_{i_3}}{a_2}} \text{ is odd}\\
                    (a_1, b_1) \cdot A_{j_2}& \text{otherwise}
                \end{cases}\\
                &\xequal{\autoref{enum:mult-4}}  \begin{cases}
                    A_{i_1}& \text{if } \order{\cpath{A_{j_2}}{A_{i_3}}{a_2}} + \order{\cpath{A_{j_1}}{A_{i_2}}{a_1}} \text{ is even }\\
                    A_{j_1}& \text{otherwise}.
                \end{cases}
            \end{align*}
            as desired.
            But notice that the path from $A_{j_1}$ to $A_{i_3}$ in the second case contains an additional $a_1 \; (=a_2)$ as it traverses the edge $(a_2, b_2)$.
            The two paths defined are either the same path (and therefore equal) or they create a cycle which (by definition) contains an even number of edges labeled $(a_1, -)$.
            In both cases we have the desired equality.
    \end{itemize}

    \textbf{$\b{xyx = xy}$}: Suppose first that $x, y \in E$, \ie $x = (A_{i_1}, A_{j_1}, (a_1, b_1))$ and $y = (A_{i_2}, A_{j_2}, (a_2, b_2))$.
    If $a_1 = a_2$, then by \autoref{enum:mult-3} we have
    \[
        (a_1, b_1) \cdot \left( (a_2, b_2) \cdot (a_1, b_1)\right) = (a_1, b_1) \cdot (a_2, b_2)
    \]
    If $a_1 \neq a_2$, then we use \autoref{enum:mult-4}.
    If $\order{\cpath{A_{j_1}}{A_{i_2}}{a_1}}$ is odd then
    \[
        (a_1 , b_1) \cdot (a_2, b_2) \cdot (a_1, b_1) = A_{i_1} \cdot (a_1, b_1) = A_{i_1} = (a_1, b_1) \cdot (a_2, b_2)
    \]
    where the second equality comes from \autoref{enum:mult-1}.
    Similarly, if $\order{\cpath{A_{j_1}}{A_{i_2}}{a_1}}$ is even then
    \[
        (a_1, b_1) \cdot (a_2, b_2) \cdot (a_1, b_1) = A_{j_1} = (a_1, b_1) \cdot (a_2, b_2)
    \]

    Now suppose $x$ is a vertex and $y$ is either an edge or a vertex.
    Then by \autoref{enum:mult-1} being used multiple times we have $xyx = xx = x = xy$.

    On the other hand, if $y = A_i$ is a vertex and $x = (a, b)$ is an edge then
    \[
        (a, b) \cdot \left( A_i \cdot (a, b) \right) = (a, b) \cdot A_i
    \]
    by \autoref{enum:mult-1} as desired.

    We therefore have $xyx = xy$ in all cases.
\end{proof}
}

\begin{remark}
    Notice that every thin LRB graph produces a rank $2$ thin MC left regular band.
    In other words, thin LRB graphs are in bijection with rank $2$ thin MC left regular bands.
    When moving from an arbitrary rank thin MC left regular band to thin LRB graphs (see \autoref{cor:to-graph}) we lose all information below the facets.
    Although it would be nice to generalize this to left regular bands of arbitrary rank, this seems to be a fairly difficult problem.
\end{remark}

\section{Further work}\label{sec:further-work}
This article presents an initial foray into thin MC left regular bands, but there are still many unanswered questions and directions one might chose to study.
In this section we outline some of these open problems.

\subsection{General left regular bands}
In \autoref{sec:mc-left-regular-bands}, we defined the adjacency graph for an arbitrary left regular band, but only studied the adjacency graph for MC left regular bands.
Removing the connected restriction does not change much from a graph perspective, but it is unclear how this affects the associated left regular band.
Understanding the structure of a left regular band whose adjacency graph is not connected would be useful as it allows us to better understand bands similar (in some sense) to the free left regular band.
Is there some decomposition of arbitrary left regular bands into left regular bands whose adjacency graph are connected?

In addition, removing the meet-semilattice condition means that there might be multiple edges between the vertices of our adjacency graph.
This adds some complexity to our graphs, but it would be worthwile to understand how this graph behaves and what the general structure of the graph is.
This would allow us to understand left regular bands from a much more general perspective.

\subsection{Higher rank}
In \autoref{sec:thin-left-regular-band-graphs} we defined thin left regular band graphs.
These turned out to be in bijection with rank $2$ thin MC left regular bands.
It is a straightforward and natural question to ask what happens when we increase the rank.
In rank $3$, we can consider the simple cycles of the graph to be elements of the left regular band (similarly to facets of a $3$-dimensional polytope).
But once we have these additional elements, it is no longer clear how multiplication should work on a global scale, especially as we continue to increase the rank.
This is slightly more manageable in the simplicial case (see the following subsection), but is still a difficult question to answer.

\subsection{Simplicial left regular bands}
Instead of generalizing, one might also restrict even further.
A \defn{simplicial} left regular band is a thin MC left regular band where every interval (in the face poset) from the identity element to a chamber is a Boolean lattice.
Simplicial left regular bands are a natural generalization of simplicial (real) hyperplane arrangements and simplicial oriented matroids.
We have some nice results coming from simplicial left regular bands, but there is still a lot of research to be done to understand these structures better.

One nice result looks towards answering the higher rank question.
Let $S$ be a rank $n$ simplicial left regular band.
The \defn{full rank $k$ subsemigroup of a left regular band $S$} (for $k \leq n$) is the subsemigroup of $S$ whose elements are precisely the elements of rank $\leq k$.
It turns out that the full rank $k$ subsemigroup of a simplicial left regular band is itself a simplicial left regular band.
This property is not necessarily true for thin MC left regular bands since, for example, a full rank $k$ subsemigroup of a thin MC left regular band is not necessarily connected nor thin (although it necessarily has a meet-semilattice face poset).
This seems to imply that simplicial left regular bands are special in some way and thus understanding their structure through their adjacency graphs would be beneficial.

\section{Acknowledgments}

The author would like to acknowledge Franco Saliola for introducing us to the topic of left regular bands and posing a problem which led to this research being carried out. 
We would also like to thank Nantel Bergeron, Marianne Johnson, Mark Kambites, Stuart Margolis, N\'ora Szak\'acs and Mike Zabrocki for many helpful conversations and ideas pertaining to left regular bands.

\printbibliography[title={References}]
\label{sec:biblio}

\end{document}